\documentclass[a4paper]{article}
\usepackage{amssymb}
\usepackage{amsthm}
\usepackage{amsmath}
\usepackage{graphicx}
\usepackage{caption}
\usepackage{subcaption}
\usepackage{float}
\usepackage{tensor}

\theoremstyle{plain}
\newtheorem{theorem}{Theorem}[section]
\newtheorem{lemma}[theorem]{Lemma}
\newtheorem{proposition}[theorem]{Proposition}
\newtheorem{problem}[theorem]{Problem}

\newtheorem{corollary}[theorem]{Corollary}

\theoremstyle{definition}
\newtheorem{definition}[theorem]{Definition}
\newtheorem{example}[theorem]{Example}

\theoremstyle{remark}
\newtheorem{remark}[theorem]{Remark}

\newcommand{\bbbn}{\mathbb{N}}

\begin{document}
\title{Strongly Polynomial Sequences as Interpretations}
\author{A.J. Goodall \thanks{Supported by grant ERCCZ LL-1201 of the Czech Ministry of Education}\\
\small Computer Science Institute of Charles University (IUUK and ITI)\\
 \small    Malostransk\' e n\' am.25, 11800 Praha 1, Czech Republic\\
 \small    {\tt andrew@iuuk.mff.cuni.cz}
 \\
\and J. Ne\v set\v ril \thanks{Supported by grant ERCCZ LL-1201 of the Czech Ministry of Education, CE-ITI P202/12/G061 of GA{\v C}R, and LIA STRUCO}\\
\small Computer Science Institute of Charles University (IUUK and ITI)\\
 \small    Malostransk\' e n\' am.25, 11800 Praha 1, Czech Republic\\
 \small    {\tt nesetril@iuuk.mff.cuni.cz}
 \\
\and P. Ossona de Mendez \thanks{Supported by grant ERCCZ LL-1201 of the Czech Ministry of Education and LIA STRUCO, and partially supported by ANR project Stint under reference ANR-13-BS02-0007}\\
\small  Centre d'Analyse et de Math\'ematiques Sociales (CNRS, UMR 8557)\\
 \small   190-198 avenue de France, 75013 Paris, France\\
\small {\tt pom@ehess.fr}
}
\date{\today}
\maketitle

\begin{abstract}
A strongly polynomial sequence of graphs $(G_n)$  is a sequence
$(G_n)_{n\in\bbbn}$ of finite graphs such that, for 
every graph $F$, the number of homomorphisms from $F$ to $G_n$ is a fixed polynomial function of $n$ (depending on $F$).
For example, $(K_n)$ is strongly polynomial since the number of homomorphisms from $F$ to $K_n$ is the chromatic polynomial of $F$ evaluated at $n$. In earlier work of de la Harpe and Jaeger, and more recently of Averbouch, Garijo, Godlin, Goodall, Makowsky, Ne\v{s}et\v{r}il, Tittmann, Zilber and others, various examples of strongly polynomial sequences and constructions for families of such sequences have been found. 

We give a new model-theoretic  method of constructing strongly polynomial sequences of graphs that uses interpretation schemes of graphs in more general relational structures. This surprisingly easy yet general method encompasses all previous constructions and produces many more. We conjecture that, under mild assumptions, all strongly polynomial sequences of graphs can be produced by the general method of quantifier-free interpretation of graphs in certain basic relational structures (essentially disjoint unions of transitive tournaments with added unary relations). We verify this conjecture for strongly polynomial sequences of graphs with uniformly bounded degree.
\end{abstract}
\\
\noindent
{\bf Subject classification (MSC 2010)} Primary: 05C31, 05C60  
Secondary: 03C13, 03C98    	

\section{Introduction}\label{sec:intro}
The chromatic polynomial $P(G,x)$ of a graph $G$, introduced by
Birkhoff over a century ago, is such that for a positive integer $n$ the value $P(G,n)$ is equal 
to the number of proper $n$-colorings of the graph $G$.
Equivalently, $P(G,n)$ is the number ${\rm hom}(G,K_n)$ of homomorphisms
from $G$ to the complete graph $K_n$. It can thus be considered that the
sequence $(K_n)_{n\in\bbbn}$ defines the chromatic polynomial by means of homomorphism counting.

A {\em strongly polynomial sequence} of graphs is a sequence
$(G_n)_{n\in\bbbn}$ of finite graphs such that, for 
every graph $F$, the number of homomorphisms from $F$ to $G_n$ is a polynomial function of $n$ (the polynomial depending on $F$ and the sequence $(G_n)$, but not on $n$). A sequence
$(G_n)_{n\in\bbbn}$ of finite graphs is {\em polynomial} if this condition holds for sufficiently large $n\geq n_F$.
The sequence of complete graphs $(K_n)$ provides a classical example of a strongly polynomial sequence. 
A homomorphism from a graph $F$ to a graph $G$ is often called a {\em $G$-colouring of $F$}, the vertices of $G$ being the ``colours" assigned to vertices of $F$  and the edges of $G$ specifying the allowed colour combinations on the endpoints of an edge of $F$.
 
The notion of (strongly) polynomial sequences of graphs was introduced by de la Harpe and Jaeger~\cite{dlHJ95} (as a generalization of the chromatic polynomial), in a paper which includes a characterization of polynomial sequences of graphs via (induced) subgraph counting and the construction of polynomial sequences by graph composition. The notion of a (strongly) polynomial sequence extends naturally to relational structures, thus allowing the use of standard yet powerful tools from model theory, like interpretations.

The ``generalized colourings" introduced in~\cite{MZ06} include only colourings invariant under all permutations of colours, which holds for $K_n$-colourings (that is, proper $n$-colourings), but not in general for $G_n$-colourings for other sequences of graphs $(G_n)$. 
However, generalized colourings in the sense defined in~\cite{MZ06} do include harmonious colourings (proper colourings with the further restriction that a given pair of colours appears only once on an edge) and others not expressible as the number of homomorphisms to terms of a graph sequence. Makowsky~\cite{M08} moves towards a classification of polynomial graph invariants, but one that does not include the class of invariants we define in this paper.

Garijo, Goodall and Ne\v{s}et\v{r}il~\cite{GGN} focus on constructing strongly polynomial sequences. 
They give a construction involving coloured rooted trees that produces a large class of strongly polynomial sequences, that in particular incorporates the Tittmann--Averbouch--Makowsky polynomial~\cite{TAM09} (not obtainable by graph composition and other previously known operations for building new polynomial sequences from old).   

We extend the scope of the term ``strongly polynomial" to sequences of general relational structures.
The property of a sequence of relational structures being strongly polynomial is preserved under a rich variety of transformations afforded by the model-theoretic notion of an interpretation scheme. We start with ``trivially" strongly polynomial sequences of relational structures, made from basic building blocks, and then by interpretation project these sequences onto graph sequences that are also strongly polynomial. 
The interpretation schemes that can be used here are wide-ranging (they need only be quantifier-free in their specification), and therein lies the power of the method. All constructions of strongly polynomial sequences that have been devised in~\cite{dlHJ95} and~\cite{GGN} are particular cases of such interpretation schemes for graph structures. 
Indeed, we conclude the paper with the conjecture that (under mild assumptions) {\em all} strongly polynomial sequences of graphs might be produced by the schema we describe here. This is verified for the case of sequences of graphs with uniformly bounded degree. 

\section{Preliminaries}
\subsection{Relational structures}
A {\em relational structure} $\mathbf A$ with {\em signature} $\lambda$ is defined by its {\em domain} $A$, a set whose elements we shall call vertices, and relations with names and arities as defined in $\lambda$. A relational structure will be denoted by an uppercase letter in boldface and its underlying domain by the corresponding lightface letter; for brevity we refer to a relational structure $\mathbf A$ with signature $\lambda$ as a {\em $\lambda$-structure}, and may just give the {\em type} (list of arities given by $\lambda$) when the symbols used for the corresponding relations are not of importance. 
A $1$-ary relation defines a subset of the domain and will be called a {\em label}, or a {\em mark} (a special type of labelling defined at the end of this section). A $2$-ary relation defines edges of a digraph on vertex set the domain, and a graph when the relation is symmetric. When the signature $\lambda$ contains only arities $1$ and $2$ we have a digraph together with labels on edges and vertices: relations of arity $1$ are labels on vertices (where in general a vertex may receive more than one label) and relations of arity $2$ are labelled edges (two vertices may be joined by edges of different labels). A Cayley graph on a group $\Gamma$ with finite generating set $S\subset\Gamma$ is an example of a labelled digraph, a directed edge joining $x$ to $y$ bearing label $s\in S$ such that $y=xs$. 

The symbols of the relations and constants defined in $\lambda$ define the non-logical symbols of the first-order language ${\rm FO}(\lambda)$ associated with $\lambda$-structures. We take first-order logic with equality as a primitive logical symbol and which is always interpreted as standard equality, so the equality relation does not appear in the signature $\lambda$. In what follows $\lambda$ will be finite, 
 in which case ${\rm FO}(\lambda)$ is countable. 
The variable symbols will be taken from the set $\{x_i : i\in \mathbb N\}$ or $\{y_i : i\in \mathbb N\}$, or, when double indexing is convenient, from $\{x_{i,j} : i,j\in \mathbb N\}$.  The subset of ${\rm FO}(\lambda)$ consisting of formulas with exactly $p$ free variables is denoted by ${\rm FO}_p(\lambda)$. 
The fragment of FO$(\lambda)$ consisting of quantifier-free formulas is denoted by QF$(\lambda)$, and ${\rm QF}_p(\lambda)$ denotes those quantifier-free formulas with exactly $p$ free variables.

For a formula $\phi\in{\rm FO}_p(\lambda)$ and a $\lambda$-structure $\mathbf A$ we define the {\em satisfaction set}
$$
\phi(\mathbf A)=\{(v_1,\dots,v_p)\in A^p:\ \mathbf A\models\phi(v_1,\dots,v_p)\},
$$
where $\phi(v_1,\dots, v_p)$ is the formula obtained upon substituting $v_i$ for each free variable $x_i$ of $\phi$, $i=1,\dots, p$. 
A {\em homomorphism} from a $\lambda$-structure $\mathbf A$ to a $\lambda$-structure $\mathbf B$ is a mapping $f:A\rightarrow B$ that preserves relations, that is, which has the property that for each relation $R$ in $\lambda$ of any given arity $r$ it is the case that $R(f(v_1),\dots, f(v_r))$ in $\mathbf B$ whenever $R(v_1,\dots, v_r)$ in $\mathbf A$. 
When $\mathbf A$ is a graph and $R$ the relation representing adjacency of vertices this is a graph homomorphism as usually defined.

The number of homomorphisms from $\mathbf A$ to $\mathbf B$ is denoted by ${\rm hom}(\mathbf A, \mathbf B)$.

A $\kappa$-structure $\mathbf A$ and a $\lambda$-structure $\mathbf B$ are {\em weakly isomorphic}, denoted by $\mathbf A\approx\mathbf{B}$, if there exists a bijection $t$ between the symbols in $\kappa$ and the symbols in $\lambda$ and a bijection $f:A\rightarrow B$ such that, for every $R\in\kappa$,
the relations $R$ and $t(R)$ have the same arity
(here denoted by $r$) and for every $v_1,\dots,v_{r}\in A$ we have
$$
\mathbf{A}\models R(v_1,\dots,v_{r})\quad\iff\quad
\mathbf{B}\models t(R)(f(v_1),\dots,f(v_{r})).
$$
In other words, $\mathbf A\approx\mathbf{B}$ if $\mathbf{A}$ and $\mathbf{B}$ are the same structure, up to renaming of the relations and relabelling of the vertices.

For signatures $\kappa$ and $\lambda$, we denote by $\kappa\sqcup\lambda$ the signature
obtained from the disjoint union of $\kappa$ and $\lambda$, 
The {\em strong sum} $\mathbf A\oplus\mathbf B$ of a
 $\kappa$-structure $\mathbf A$ and a $\lambda$-structure $\mathbf B$ is the $\kappa\oplus\lambda$-structure
whose domain is the disjoint union $A\sqcup B$ of the domains of $\mathbf{A}$ and $\mathbf{B}$, where
for every $R\in\kappa$ and $S\in\lambda$ (with respective arities $r$ and $s$) and for every $v_1,\dots,v_{\max(r,s)}$ in $A\sqcup B$ it holds that
\begin{align*}
\mathbf A\oplus\mathbf B\models R(v_1,\dots,v_r)\quad&\iff\quad (v_1,\dots,v_r)\in A^r\text{ and }\mathbf{A}\models R(v_1,\dots, v_r),\\
\mathbf A\oplus\mathbf B\models S(v_1,\dots,v_s)\quad&\iff\quad (v_1,\dots,v_s)\in B^s\text{ and }\mathbf{S}\models S(v_1,\dots, v_s).
\end{align*}
Note that the strong sum is not commutative, but we do have
$$\mathbf{A}\oplus\mathbf{B}\approx\mathbf{B}\oplus\mathbf{A}.$$

A class $\mathcal C$ of $\lambda$-structures is {\em marked} by a relation $U\in\lambda$ if $U$ is unary, and for every $\mathbf A\in\mathcal C$ we have
$$\mathbf A\models (\forall x)\ U(x).$$

\subsection{Sequences of relational structures}

We begin with a definition of the notion that is the subject of this paper. 
\begin{definition}\label{def:polyseq}
A sequence $(\mathbf A_n)_{n\in\bbbn}$ of $\lambda$-structures is 
	{\em strongly polynomial} if for every quantifier-free formula $\phi$ there
	is a polynomial $P_\phi$ such that $|\phi(\mathbf A_n)|=P_\phi(n)$ holds for every $n\in\bbbn$.
\end{definition}
\begin{remark} A sequence is {\em polynomial} if for every quantifier-free formula $\phi$ there
	is a polynomial $P_\phi$ and an integer $n_\phi$ such that 
	$|\phi(\mathbf A_n)|=P_\phi(n)$ holds for every integer $n\geq n_\phi$. We shall only consider strongly polynomial  sequences in this paper, but occasionally it will help to clarify what is involved in the property of being a strongly polynomial sequence by giving examples of sequences that are polynomial but not strongly polynomial. 
	\end{remark}

We begin by formulating equivalent criteria for a sequence of structures to be strongly polynomial in terms of homomorphisms or induced substructures (Theorem~\ref{thm:equiv} below). For graph structures this will make a direct connection to the notion as originally defined by de la Harpe and Jaeger~\cite{dlHJ95}.

For this we require the following lemma: 

\begin{lemma}
\label{lem:QFhom}
Let $\lambda$ be a signature for relational structures. For every formula $\phi$ in ${\rm QF}(\lambda)$, 
there exist  $\lambda$-structures $\mathbf F_1,\dots,\mathbf F_\ell$ and integers
$c_1,\dots,c_\ell$ such that for every  $\lambda$-structure $\mathbf A$ we have
$$|\phi(\mathbf A)|=\sum_i c_i\,{\rm hom}(\mathbf F_i,\mathbf A).$$
\end{lemma}
\begin{proof}
Let $\phi\in {\rm QF}(\lambda)$ be quantifier-free with free variables $x_1,\dots, x_p$. 

We first put $\phi$ in disjunctive normal form, with basic terms $(x_i=x_j)$ (for $1\leq i<j\leq p$) and 
$R(x_{i_1},\dots,x_{i_r})$ for relation $R$ in $\lambda$ with arity $r$ and $i_1,\dots,i_r$ in $\{1,\dots,p\}$.
Thus $\phi$ is logically equivalent to 
\begin{equation}
\label{eq:disj1}
\bigvee_\mathcal{P} \zeta_\mathcal{P}\wedge\phi_\mathcal{P},
\end{equation}
where the disjunction runs over partitions $\mathcal P$ of $\{1,\dots,p\}$, where $\zeta_\mathcal{P}$ is the conjunction of all equalities and non-equalities that have to hold between free variables $x_1,\dots,x_p$ in order that $\mathcal P$ induces the partition of free variables into their $k$ ($1\leq k\leq p$) equality classes, and where $\phi_\mathcal{P}$ is a formula with $k$ free variables defining the
$\lambda$-structure $\mathbf F_\mathcal{P}$ induced by $x_{i_1},\dots,x_{i_k}$ for arbitrary choice of representatives $i_1,\dots,i_k$ of the parts of $\mathcal{P}$. 
As all the terms in~\eqref{eq:disj1} are mutually exclusive, we have
$$|\phi(\mathbf A)|=\sum_\mathcal{P}{\rm ind}(\mathbf F_\mathcal{P},\mathbf A),$$
where ${\rm ind}(\mathbf F,\mathbf{A})$ denotes the number of injective mappings $f:F\rightarrow A$ defining an isomorphism between $\mathbf F$ and its image.

 We wish to rewrite this sum in terms of induced substructures as one in terms of homomorphism numbers, and we achieve this in two steps. First we move from counting induced substructures to counting injective homomorphisms, 
$${\rm inj}(\mathbf F,\mathbf A)=\sum_{\stackrel{\mathbf F': F'=F}{\forall R\in\lambda\; R(\mathbf F')\supseteq R(\mathbf F)}}{\rm ind}(\mathbf F',\mathbf A),$$ 
in which ${\rm inj}(\mathbf F, \mathbf A)$ denotes the number of injective homomorphisms from $\mathbf F$ into $\mathbf A$ and $R(\mathbf F)=\{(v_1,\dots, v_r)\in F^r:R(v_1,\dots, v_r)\}$ is the set of tuples satisfying the $r$-ary relation $R$ in $\mathbf F$, and similarly $R(\mathbf F')$ denotes those tuples satisfying the relation $R$ in $\mathbf F'$. From this identity, by inclusion-exclusion we obtain 
$${\rm ind}(\mathbf F,\mathbf A)=\sum_{\stackrel{\mathbf F':F'=F}{\forall R\in\lambda\; R(\mathbf F')\subseteq R(\mathbf F)}}\left(\prod_{R\in\lambda}(-1)^{|R(\mathbf F)|\!-\!|R(\mathbf F')|}\right){\rm inj}(\mathbf F',\mathbf A).$$
The second step is to move from counting injective homomorphisms to counting all homomorphisms, the relationship between which is given by $${\rm hom}(\mathbf F, \mathbf A)=\sum_{\Theta}{\rm inj}(\mathbf{F}/\Theta,\mathbf A),$$
where the sum is over partitions $\Theta$ of the domain $F$ of $\mathbf F$ and $\mathbf F/\Theta$ is the structure obtained from $\mathbf F$ by identifying elements of its domain $F$ that lie in the same block of $\Theta$. We then obtain
$${\rm inj}(\mathbf F,\mathbf A)=\sum_{\Theta}\mu(\Theta){\rm hom}(\mathbf F/\Theta,\mathbf A),$$
where $$\mu(\Theta)=\prod_{I\in\Theta}(-1)^{|I|\!-\!1}(|I|\!-\!1)!$$ is the M\"obius function of the lattice of partitions of $F$. The statement of the lemma now follows.
\end{proof}
\begin{remark} In the context of graphs (structures with signature comprising a symmetric binary relation) the identities used in the proof of Lemma~\ref{lem:QFhom} between counts of induced subgraphs, homomorphisms and injective homomorphisms find widespread application (see for example~\cite{BCLSV06}).
\end{remark}

We now come to the promised reformulation of the notion of a strongly polynomial sequence of structures.
\begin{theorem}
\label{thm:equiv}
The following are equivalent for a sequence $(\mathbf A_n)_{n\in\bbbn}$ of $\lambda$-structures:
\begin{itemize}
\item[(i)] The sequence $(\mathbf A_n)_{n\in\bbbn}$ is strongly polynomial;
\item[(ii)] For each quantifier-free formula $\phi$ there is a polynomial $P_\phi$ such that $|\phi(\mathbf A_n)|=P_\phi(n)$ for each $n\in\mathbb{N}$;
\item[(iii)] For each $\lambda$-structure $\mathbf F$ there is a polynomial $P_{\mathbf F}$ such that ${\rm hom}(\mathbf F,\mathbf A_n)=P_{\mathbf F}(n)$ for each $n\in\mathbb{N}$;
\item[(iv)] For each $\lambda$-structure $\mathbf G$ there is a polynomial $P_{\mathbf G}$ such that ${\rm ind}(\mathbf G,\mathbf A_n)=P_{\mathbf G}(n)$ for each $n\in\mathbb{N}$.
\end{itemize}
\end{theorem}
\begin{proof} Items (i) and (ii) are equivalent by definition. As homomorphisms and finite induced substructures can be expressed in QF, items (iii) and (iv) are both special cases of (ii). Finally, Lemma~\ref{lem:QFhom} shows that (iii) implies (ii), and the proof of the same lemma that (iv) also implies (ii).  
\end{proof}

\begin{remark}
\label{rem:rat}
In Theorem~\ref{thm:equiv}, the equivalence also holds with weaker conditions in which the existence of a polynomial function is replaced by the existence of a rational function. Indeed, assume $f(x)=P(x)/Q(x)$ is a rational function, where $P$ and $Q$ are  polynomials.  Then there exist polynomials $R(x), S(x)$ such that $f(x)=S(x)+R(x)/Q(x)$ and $\deg R<\deg Q$. For sufficiently large $n$, it follows that $-1<R(x)/Q(x)<1$. As $f(x)$ takes only integral values on integers, it follows that $R(n)/Q(n)=0$ for sufficiently large $n$. Hence $R=0$ and $f$ is a polynomial function.
\end{remark}
\begin{remark}
\label{rem:extsig}
Assume $\kappa$ is a signature that is a subset of another signature $\lambda$. Then every $\kappa$-structure can be considered as a $\lambda$-structure. The notion of strongly polynomial sequence is robust in the sense that
a sequence $(\mathbf A_n)_{n\in\bbbn}$ of $\kappa$-structures is strongly polynomial if and only if $(\mathbf A_n)_{n\in\bbbn}$ (considered as a sequence of $\lambda$-structures) is strongly polynomial: indeed, for every $\lambda$-structure $\mathbf F$, either $\mathbf F$ contains a relation not in $\kappa$ and thus
${\rm hom}(\mathbf F,\mathbf A_n)=0$ for every $n\in\bbbn$, or $\mathbf F$ can be considered as a $\kappa$-structure, and the number of homomorphisms from $\mathbf F$ to $\mathbf A_n$ does not depend on the signature considered.
\end{remark}

We end this section with a few statements on the invariance of strongly polynomial sequences with respect to various operations.

\begin{lemma}
\label{lem:union}
Let $(\mathbf A_{i,n})_{n\in\bbbn}$  be strongly polynomial sequences of $\lambda_i$-structures, for $1\leq i\leq k$.
Then the sequence $(\bigoplus_{i=1}^k \mathbf A_{i,n})_{n\in\bbbn}$ is strongly polynomial.
\end{lemma}
\begin{proof}
By Lemma~\ref{lem:QFhom} it is sufficient to check polynomiality of 
${\rm hom}(\mathbf F,\bigoplus_{i=1}^k \mathbf A_{i,n})$. Let $\mathbf F_1,\dots,\mathbf F_\ell$ be the connected components
of $\mathbf F$. Then the proof follows from the identity
\begin{align*}
{\rm hom}(\mathbf F,\bigoplus_{i=1}^k \mathbf A_{i,n})&=\prod_{j=1}^\ell {\rm hom}(\mathbf F_j,\bigoplus_{i=1}^k \mathbf A_{i,n})\\
&=\prod_{j=1}^\ell \sum_{i=1}^k {\rm hom}(\mathbf F_j,\mathbf A_{i,n}).
\end{align*}
(Note that in the last equality we consider each $\mathbf A_{i,n}$ as a $\bigsqcup_{i=1}^k\lambda_i$-structure, which it is safe to do according to Remark~\ref{rem:extsig}.) 
\end{proof}


\begin{lemma}
\label{lem:multP}
Let $(\mathbf A_n)_{n\in\bbbn}$ be a strongly polynomial sequence of $\lambda$-structures and let $P$ be a polynomial such that $P(n)\in\mathbb{N}$ for $n\in\mathbb{N}$.
Then the sequence $(P(n)\,\mathbf A_n)_{n\in\bbbn}$ is strongly polynomial, where $P(n)\,\mathbf A_n$ denotes the $\lambda$-structure obtained as the disjoint union of $P(n)$ copies of $\mathbf A_n$.
\end{lemma}
\begin{proof}
By Lemma~\ref{lem:QFhom} it is sufficient to check polynomiality of 
${\rm hom}(\mathbf F,P(n)\mathbf A_n)$. Let $\mathbf F_1,\dots,\mathbf F_\ell$ be the connected components
of $\mathbf F$. Then the proof follows from the identity
$${\rm hom}(\mathbf F,P(n)\,\mathbf A_n)=\prod_{j=1}^\ell{\rm hom}(\mathbf F_j,P(n)\,\mathbf A_n)
=\prod_{j=1}^\ell(P(n)\,{\rm hom}(\mathbf F_j,\mathbf A_n)).
$$
\end{proof}

\begin{lemma}
\label{lem:extractP}
Let $(\mathbf A_n)_{n\in\bbbn}$ be a strongly polynomial sequence of $\lambda$-structures and let $P$ be a polynomial such that $P(n)\in\mathbb{N}$ for $n\in\mathbb{N}$.
Then the sequence $(\mathbf A_{P(n)})_{n\in\bbbn}$ is strongly polynomial.
\end{lemma}
\begin{proof}
For every $\phi\in{\rm QF}(\lambda)$ there is a polynomial $Q$ such that
$|\phi(\mathbf A_n)|=Q(n)$ for each $n\geq 1$, as $(\mathbf A_n)_{n\in\bbbn}$ is a strongly polynomial sequence.
Thus $|\phi(\mathbf A_{P(n)})|=Q\circ P(n)$.
It follows that the sequence $(\mathbf A_{P(n)})_{n\in\bbbn}$ is strongly polynomial.
\end{proof}

\section{Basic structures}
\label{sec:basic}
Two special types of marked structures will be of particular interest in this paper:
\begin{itemize}
\item $\mathbf E$ is a structure whose domain is a singleton, and whose signature is a single unary relation $U$ that satisfies $(\forall x)\ U(x)$;
\item $\mathbf T_n$ is a transitive tournament of order $n$. Precisely, the domain of $\mathbf T_n$ is $[n]=\{1,\dots,n\}$ and its signature contains a single binary relation $S$ with 
$\mathbf{T}_n\models S(i,j)\; \iff\; i<j$, and a single unary relation $U$, which satisfies $(\forall x)\ U(x)$.
\end{itemize}

\begin{definition}\label{def:basic}
A {\em basic structure with parameters} $(k,\ell)$ is a structure 
$$\mathbf{B}=\overbrace{\mathbf E\oplus\dots\oplus\mathbf E}^{\ell\text{ times}}\oplus \overbrace{\mathbf T_{N_1}\oplus\dots\oplus\mathbf T_{N_k}}^{k\text{ times}}.$$
which is the strong sum of $\ell$ marked vertices
and $k$ marked transitive tournaments 
of respective orders $N_1(\mathbf{B}),\dots,N_k(\mathbf{B})$.
We denote by $\mathcal B_{k,\ell}$ the class
of all basic structures with parameters $(k,\ell)$ and
by $\beta_{k,\ell}$ the signature of these structures. It will be notationally convenient to 
assume that the relations in $\beta_{k,\ell}$ are
$U_1^E,\dots,U_\ell^E$, $U_1^T,\dots, U_k^T$, $S_1,\dots,S_k$. 


A {\em basic sequence} 
 is a sequence $(\mathbf B_n)_{n\in\bbbn}$ of basic structures 
$\mathbf B_n\in\mathcal B_{k,\ell}$ (for some fixed $k,\ell\in\mathbb{N}$) such that 
there are non-constant polynomials $Q_i$, $1\leq i\leq k$ with $Q_i(n)=N_i(\mathbf{A}_n)$
(for every $1\leq i\leq k$ and $n\in\bbbn$).
\end{definition}

It follows directly from Lemmas~\ref{lem:extractP} and~\ref{lem:union} that every basic sequence is strongly polynomial. 

\section{Strongly polynomial sequences by interpretations}
The basic building blocks we use for constructing strongly polynomial graph sequences are marked tournaments $(\mathbf{T}_{P(n)})$ on a polynomial number of vertices, and the constant sequence $(\mathbf E)$ consisting of a single marked vertex. From these we can produce all the strongly polynomial sequences given in~\cite{dlHJ95} and~\cite{GGN} and much more. (In Section~\ref{sec:ex} we give a large selection of examples of strongly polynomial graph sequences that exhibits their diversity.) To do this we need just two operations: strong sum and graphical interpretation of structures. The latter is a potent operation for it produces graph sequences from strongly polynomial sequences of $\lambda$-structures of arbitrary signature $\lambda$, while strong sum is an essential operation for gluing together separately constructed sequences, from which a sequence of larger structures can be made. 
(Note that one could equivalently consider disjoint union and QF-interpretation in place of strong sum and QF-interpretation.) 
\subsection{Interpretation schemes}
We begin with the definition of an interpretation scheme that we shall require. 

\begin{definition}\label{def:int}
Let $\kappa,\lambda$ be signatures, where the signature $\lambda$ has $q$ relational symbols $R_1,\dots, R_q$ with respective arities $r_1,\dots, r_q$. An {\em interpretation scheme} $I$ of $\lambda$-structures in $\kappa$-structures with exponent $p$ is a tuple $I=(p,\rho_0,\dots,\rho_q$), where $p$ is a positive integer, $\rho_0\in{\rm FO}_p(\kappa)$, and $\rho_i\in{\rm FO}_{pr_i}(\kappa)$, for $1\leq i\leq q$.

 For $\kappa$-structure $\mathbf A$, we denote by $I(\mathbf A)$ the $\lambda$-structure $\mathbf B$ with  domain 
$B=\rho_0(\mathbf A)$ and relations defined by  
$$\mathbf B\models R_i(\mathbf{v}_1,\dots, \mathbf{v}_{r_i})\quad\Longleftrightarrow \quad\mathbf A\models\rho_i(\mathbf{v}_1,\dots, \mathbf{v}_{r_i})$$
(for $1\leq i\leq q$ and $\mathbf{v}_1,\dots, \mathbf{v}_{r_i}\in B$).
\end{definition}

\begin{definition}\label{def:QFint}
A {\em QF-interpretation scheme} is an interpretation scheme in which all the formulas $\rho_i, 0\leq i\leq q$, used to define it in Definition~\ref{def:int} are quantifier-free.
\end{definition}

\begin{example}
Let us consider two signatures $\kappa$ and $\lambda$ with $\kappa\subset\lambda$. Then the following transformations are easily (and almost trivially) checked to be definable by QF-interpretation schemes:
\begin{itemize}
\item {\em Lift}, the canonical injection of $\kappa$-structures into a $\lambda$-structures (same relations);
\item {\em Forget}, the canonical projection of $\lambda$-structures onto $\kappa$-structures (filters out relations not in $\kappa$);
\item {\em Merge}, which maps $\lambda\sqcup\lambda$-structures into $\lambda$-structures by merging similar relations (so that ${\rm Merge}(\mathbf{A}\oplus\mathbf{B})=\mathbf{A}+\mathbf{B}$);
\item {\em Mark}, which maps $\lambda$-structures into $\lambda^+$-structures (where $\lambda^+$ is the signature obtained by adding to $\lambda$ a new unary relation $U$) by putting every element in the relation $U$.
\end{itemize}
\end{example}


Since our goal is to construct strongly polynomial sequences of graphs, we shall have a particular use for interpretation schemes of graph structures in $\kappa$-structures. 
\begin{definition}\label{def:graphint}
A {\em graphical interpretation scheme} $I$ of $\kappa$-structures is a triple $(p,\iota,\rho)$, where $p$ is a positive integer, $\iota\in{\rm FO}_p(\kappa)$, and $\rho\in{\rm FO}_{2p}(\kappa)$ is symmetric (that is, such that $\vdash \phi(x,y)\leftrightarrow\phi(y,x)$).
For every $\kappa$-structure $\mathbf A$, 
the interpretation $I(\mathbf A)$ has vertex set 
$V=\iota(\mathbf A)$
and edge set 
$$E=\{\{\mathbf u,\mathbf v\}\in V\times V:\ \mathbf A\models \rho(\mathbf u,\mathbf v)\}.$$
\end{definition}

We have already mentioned a graphical interpretation scheme of digraph structures: that which interprets an orientation of a graph as the underlying undirected graph simply by forgetting the edge directions (for example $K_n$ from  $\mathbf T_n$). Taking the complement of a graph $G$ is a graphical interpretation scheme (of graph structures) with $p=1$ 
in which we take $\iota=1$ (constantly true), and $\rho(x,y)=\neg R(x,y)$, where $R(x,y)$ represents adjacency between $x$ and $y$. The square of the graph $G$, joining vertices $x$ and $y$ when they are adjacent or share a common neighbour, is a graphical interpretation scheme (of graph structures) with $p=1$, $\iota=1$, and $\rho(x,y)=R(x,y)\vee (\exists{z}\; R(x,z)\wedge R(z,y))$ (this one requires a quantifier). 
The line graph of a simple undirected graph $G$ can be realized indirectly: orient the edges $G$ and use a graphical interpretation scheme of digraph structures with $p=2$, by taking $\iota(x,y)=R(x,y)$, where $R$ is the antisymmetric relation representing oriented edges of $G$, and $\rho(x_1,y_1,x_2,y_2)=[(x_1=y_2)\vee (y_1=y_2)\vee (x_1=y_2)\vee (x_2=y_1)]\wedge \neg[(x_1=x_2)\wedge(y_1=y_2)]\wedge\neg[(x_1=y_2)\wedge(x_2=y_1)]$. A more natural way to define a graphical interpretation scheme (of graph structures) for the operation of taking the line graph requires the general interpretation schemes discussed in Section~\ref{sec:Discussion}. Compare  also Remark~\ref{rem:clique_intersect} in Section~\ref{sec:Johnson}.

The following standard result from model theory (see for example~\cite[Section 3.4]{L09}) underlies the key role interpretation will play in moving from one strongly polynomial sequence of structures to another.

Let $I=(p,\rho_0,\dots,\rho_q)$ be an interpretation scheme of $\lambda$-structures in $\kappa$-structures.
We inductively define the mapping $M_I$ from terms in  ${\rm FO}(\lambda)$ to terms in  ${\rm FO}(\kappa)$ by:
\begin{itemize}
\item $M_{I}(x_i)=(x_{i,1},\dots,x_{i,p})$ for variable $x_i$;
\item $M_{I}(R_i(t_1,\dots,t_{r_i}))=
\rho_i(M_{I}(t_{1}),\dots,M_{I}(t_{r_i}))$;
\item $M_{I}(\phi\vee\psi)=M_{I}(\phi)\vee M_{I}(\psi)$;
\item $M_{I}(\phi\wedge\psi)=M_{I}(\phi)\wedge M_{I}(\psi)$;
\item $M_{I}(\neg\phi)=\neg M_{I}(\phi)$;
\item $M_{I}((\exists x)\ \phi)=(\exists x_1\dots\exists x_p)\ \bigwedge_{i=1}^p\rho_0(x_i)\wedge M_{I}(\phi)$;
\item $M_{I}((\forall x)\ \phi)=(\forall x_1\dots\forall x_p)\ \bigwedge_{i=1}^p\rho_0(x_i)\rightarrow M_{I}(\phi)$.
\end{itemize}
Then we define the mapping $\tilde{I}:{\rm FO}(\lambda)\rightarrow{\rm FO}(\kappa)$ by
$$\tilde{I}(\phi)=\bigwedge_{i=1}^k \rho_0(x_i)\,\wedge\,M_I(\phi),$$
where $x_1,\dots,x_k$ are the free variables of $\phi$.
 
Note that if all the $\rho_i$ are quantifier-free then
$\tilde{I}$ maps quantifier-free formulas to quantifier-free formulas.

\begin{lemma}\label{lem:dual} If $I$ is an interpretation scheme of $\lambda$-structures in $\kappa$-structures then for every $\phi\in{\rm FO}_r(\lambda)$ and every $\kappa$-structure~$\mathbf A$ we have
$$\phi(I(\mathbf A))=\tilde{I}(\phi)(\mathbf{A}).$$
\end{lemma}

As a corollary of Lemma~\ref{lem:dual} we have

\begin{corollary}\label{cor:IntPreserveSP}
If $(\mathbf A_n)_{n\in\bbbn}$ is a strongly polynomial sequence of $\kappa$-structures and if
$I$ is a QF-interpretation scheme of $\lambda$-structures in $\kappa$-structures, then $(I(\mathbf A_n))_{n\in\bbbn}$ is a strongly polynomial sequence of $\lambda$-structures.
\end{corollary}

The interest of marked structures is that marking provides a way to track components of strong sums, thus allowing the combination of componentwise defined interpretation schemes, as shown by the next lemma.

\begin{lemma}
\label{lem:mergeInt}
Let $I_i$ ($1\leq i\leq k$) be interpretation schemes of $\lambda_i$-structures in $\kappa_i$-structures with exponent $p_i$.
Assume that each $\kappa_i$ contains a unary relation $U_i$.

Then there exists an interpretation $I$ scheme of $\bigsqcup_{i=1}^k\lambda_i$-structures in $\bigsqcup_{i=1}^k\kappa_i$-structures
with exponent $p=\max p_i$ such that, for every  $\kappa_i$-structure $\mathbf{A}_i$ marked by $U_i$ ($1\leq i\leq k$), we have
$$I\bigl(\bigoplus_{i=1}^k\mathbf{A}_i\bigr)=\bigoplus_{i=1}^k I_i(\mathbf{A}_i).$$
Moreover, if all the $I_i$'s are QF-interpretation schemes, we can require $I$ to be a QF-interpretation scheme.
\end{lemma}
\begin{proof}
For $1\leq i\leq k$, let 
 $I_i=(p_i,\rho_0^i,\dots,\rho^i_{q_i})$, where $\rho^i_j$ has $r_{i,j}p_i$ free variables ($1\leq j\leq q_i$). We define the interpretation scheme
$I=(p,\rho_0,\rho_{i,j})$ ($1\leq i\leq k$, $1\leq j\leq q_i$) 
with exponent $p$ 
of $\bigsqcup_{i=1}^k\lambda_i$-structures in $\bigsqcup_{i=1}^k\kappa_i$-structures as follows:
the formula $\rho_0$ is
$$
\rho_0:\ \bigvee_{i=1}^k \biggl(
\bigwedge_{j=1}^{p_i}U_i(x_j)\wedge
\bigwedge_{j=p_i}^{p}(x_j=x_p)\wedge
\rho_{0}^i(x_1,\dots,x_{p_i})
\biggr)
$$
and the formula $\rho_{i,j}$ (with $r_{i,j}p$ free variables) is defined by 
$$
\rho_{i,j}:\ \bigwedge_{\ell=1}^{r_{i,j}p}U_i(x_\ell)\wedge\rho^i_j(x_1,\dots,x_{p_i},
x_{p+1},\dots,x_{p+p_i},\dots,x_{(r_{i,j}-1)p+1},\dots,x_{(r_{i,j}-1)p+p_i}).
$$
Then $I$ obviously satisfies the requirements of the lemma statement.
\end{proof}

QF-interpretation schemes of strong sums are instrumental in the construction of strongly polynomial sequences, as exemplified by the next result.


\begin{corollary}\label{cor:prod}
\label{lem:const}
Let $(\mathbf A_{n})_{n\in\bbbn}$ and $(\mathbf B_{n})_{n\in\bbbn}$ be strongly polynomial sequences of graphs.
Then $(\mathbf A_{n}+\mathbf B_{n})_{n\in\bbbn}$, $(\mathbf A_{n}\times \mathbf B_{n})_{n\in\bbbn}$, $(\mathbf A_{n}\Box\,\mathbf B_{n})_{n\in\bbbn}$,
$(\mathbf A_{n}\boxtimes\,\mathbf B_{n})_{n\in\bbbn}$,
and $(\mathbf A_{n}[\mathbf B_{n}])_{n\in\bbbn}$ are strongly polynomial sequence of graphs (formed respectively by disjoint union, direct product, Cartesian product, strong product, and lexicographic product).
\end{corollary}
\begin{proof}
This follows from Corollary~\ref{cor:IntPreserveSP} and Lemma~\ref{lem:union}, by noticing that all the constructions listed here are QF-interpretations of ${\rm Mark}(\mathbf A_n)\oplus{\rm Mark}(\mathbf B_n)$.
\end{proof}


Many more constructions can be used to combine strongly polynomial sequences of structures by means of strong sum and QF-interpretations, of which the following is an example.

\begin{example}
Let $t$ be a fixed odd integer, and 
let $(\mathbf A_{i,n})_{n\in\bbbn}$ ($1\leq i\leq t$) be a strongly polynomial sequences of graphs. We define $\mathbf B_n$ as the graph with vertex set 
$A_{1,n}\times\dots\times A_{t,n}$ where 
$(u_1,\dots,u_t)$ is adjacent to $(v_1,\dots,v_t)$ in $\mathbf B_n$ if there is a majority of $i\in\{1,\dots,t\}$ such that $u_i$ is adjacent to $v_i$ in $\mathbf A_{i,n}$. Then $(\mathbf{B}_n)_{n\in\bbbn}$ is a strongly polynomial sequence.
\end{example}

\section{Interpretations of basic sequences}

As already noted in Section~\ref{sec:basic}, every basic sequence is strongly polynomial. It follows from Lemma~\ref{cor:IntPreserveSP} that this is also the case for their QF-interpretations:

\begin{corollary}
\label{cor:intBasic}
If $(\mathbf B_n)$ is a basic sequence and $I$ is a QF-interpretation of $\lambda$-structures in $\beta_{k,\ell}$-structures, then $(I(\mathbf B_n))$ is a strongly polynomial sequence of $\lambda$-structures.
\end{corollary}

The class $\mathcal P$ of sequences $(\mathbf A_n)_{n\in\bbbn}$ that can be obtained by QF-interpretations of basic sequences is quite rich. In particular, it is closed under the following operations, which we know can be used to construct new strongly polynomial sequences from old ones:
\begin{itemize}
\item Extracting a subsequence $(\mathbf A_{P(n)})$, where $P$ is a polynomial such that $P(n)\in\mathbb{N}$ for $n\in\mathbb{N}$ (as $(\mathbf B_{P(n)})_{n\in\bbbn}$ is a basic sequence);
\item Applying a QF-interpretation scheme (as the composition of two QF-interpretation schemes defines a QF-interpretation scheme);
\item Strong sums (according to Lemma~\ref{lem:mergeInt});
\item  Multiplying by a polynomial $P$ such that $P(n)\in\mathbb{N}$ for $n\in\mathbb{N}$ (as $P(n)\,\mathbf A_n$ is a QF-interpretation of $\mathbf T_{P(n)}\oplus\mathbf A_n$).
\end{itemize}

\subsection{Concrete examples}\label{sec:ex}
\subsubsection{Crown graphs}
We start from the basic sequence $\mathbf A_n=\mathbf E\oplus\mathbf E\oplus \mathbf T_n$.
Consider the graphical interpretation scheme $I=(2,\iota,\rho)$, where
\begin{align*}
\iota(x_1,x_2):\quad&  U^T_1(x_1)\wedge\neg U^T_1(x_2)\\\rho(x_1,x_2,y_1,y_2): &\neg(x_1=y_1)\wedge \neg(U_1^E(x_2)\leftrightarrow U^E_1(y_2)).
\end{align*}
Then the graph obtained is the {\em crown graph} $S_n$ ($K_{n,n}$ minus a perfect matching), and it follows that $(S_n)_{n\in\bbbn}$ is a strongly polynomial sequence.
Similarly, for every integer $k$, the sequence $({\rm KG}_{n,k})_{n\in\bbbn}$ of the {\em Kneser graphs} is strongly polynomial.

\subsubsection{Generalized Johnson graphs}\label{sec:Johnson}
We start from the basic sequence $\mathbf A_n=\mathbf T_n$ and consider, for fixed integer $k$ and subset $D\subseteq [k]$, the graphical interpretation scheme $I=(k,\iota,\rho)$, where
\begin{align*}
\iota(x_1,\dots,x_k): &\bigwedge_{i=1}^{k-1}S_1(x_i,x_{i+1})\\
\rho(x_1,\dots,x_k,y_1,\dots,y_k): 
&\bigvee_{\substack{I,J\subseteq [k]\\ |I|=|J|\\
|I|\in D}}\biggl(
\bigwedge_{i\not\in I, j\not\in J}\neg(x_i=y_j)\,\wedge\,
\bigwedge_{i\in I}\bigvee_{j\in J} (x_i=y_j)\biggr)
\end{align*}
Then $I(\mathbf A_n)$ is the generalized Johnson graph $J_{n,k,D}$, which is the graph
with vertices $\binom{[n]}{k}$ and where $X$ and $Y$ are adjacent whenever $|X\cap Y|\in D$.
\medskip

\begin{remark}\label{rem:clique_intersect}
Similarly, let $k$ be an integer and let $D\subseteq [k]$, and let $(G_n)_{n\in\bbbn}$ be a
 strongly polynomial sequence of graphs.
For $n\in\bbbn$ define $H_n$ as the graph whose vertices are the $k$-cliques of $G_n$, where two $k$-cliques of $G_n$
are adjacent in $H_n$ if the cardinality of their intersection belongs to $D$. Then
the sequence of graphs $(H_n)_{n\in\bbbn}$ is strongly polynomial. In particular,
the sequence $(L(G_n))_{n\in\bbbn}$ of  line graphs is strongly polynomial.
\end{remark}

\subsubsection{Vertex-blowing of a fixed graph}
Let $F$ be a fixed graph with vertex set $[k]$. To each vertex $i$ of $F$ is associated a polynomial $P_i$ such that $P_i(n)\in\mathbb{N}$ for $n\in\mathbb{N}$. 
Let $\mathbf A_n=\bigoplus_{i=1}^k\mathbf{T}_{P_i(n)}$.
We define the graphical interpretation scheme
$I=(1,\iota,\rho)$ by
\begin{align*}
\iota(x):\quad&1\\
\rho(x,y):&\bigvee_{ij\in E(F)} U^T_i(x)\wedge U^T_j(y)
\end{align*}
Then $I(\mathbf A_n)$ is the vertex-blowing of $F$, in which vertex $i$ is replaced by $P_i(n)$ twin copies of $i$.

\subsubsection{Tree-blowing of a fixed rooted tree}
Let $F$ be a rooted tree with edge set $E=\{2,\dots,k\}$.
To each edge $e\in E$ is associated a polynomial $P_e$ such that $P_e(n)\in\mathbb{N}$ for $n\in\mathbb{N}$ and let $P_1$ be another such polynomial (for the root).
Let $\mathbf A_n=\bigoplus_{i=1}^k\mathbf T_{P_i(n)}$.
An $F$-path will be a sequence $(1,e_1,\dots,e_i)$ corresponding to a path from the root of $F$.
Define the graphical interpretation scheme $I=(k,\iota,\rho)$ by
\begin{align*}
\iota(x_1,\dots,x_k):\quad&
\bigvee_{\text{$F$-path }(a_1,\dots,a_t)}
\biggl(\bigwedge_{i=1}^t U^T_{a_i}(x_i)\,\wedge\,
\bigwedge_{i=t+1}^k (x_i=x_t)\biggr)\\
\rho(x_1,\dots,x_k,y_1,\dots,y_k):\ 
&\rho'(x_1,\dots,x_k,y_1,\dots,y_k)\vee
\rho'(y_1,\dots,y_k,x_1,\dots,x_k)
\intertext{where}
\rho'(x_1,\dots,x_k,y_1,\dots,y_k):\\
&\hspace{-1cm}\bigvee_{i=1}^{k-1}\Bigl(\bigwedge_{j=1}^i (x_j=y_j)\,\wedge(x_i=x_k)\wedge\neg(y_i=y_k)\wedge(y_{i+1}=y_k)\Bigr)
\end{align*}
Then $I(\mathbf A_n)$ is the tree-blowing of $F$ (in~\cite{GGN} this operation on rooted trees is called ``branching"). 

\subsubsection{Union of stars of orders $1,\dots,P(n)$}
We consider the graphical interpretation scheme $I=(2,\iota,\rho)$ defined by
\begin{align*}
\iota(x,y):&S_1(y,x)\\
\rho(x_1,y_1,x_2,y_2):&(y_1=y_2)\wedge
[(x_1=y_1)\wedge S_1(x_2,y_2)\,\vee\,
(x_2=y_2)\wedge S_1(x_1,y_1)]
\end{align*}
Then, for $\mathbf A_n=\mathbf T_{P(n)}$, we have
$$I(\mathbf{A_n})=\bigcup_{i=1}^{P(n)}\mathbf S_i,$$
where $\mathbf{S}_i$ is the star of order $i$. 

\subsubsection{Half graphs}
Let $\mathbf A_n=\mathbf{E}\oplus\mathbf E\oplus \mathbf T_n$.
Consider the graphical interpretation scheme $I=(2,\iota,\rho)$ where:
\begin{align*}
\iota(x_1,x_2):\quad&U^T_1(x_1)\wedge\neg U^T_1(x_2)\\
\rho(x_1,x_2,y_1,y_2): 
&S_1(x_1,y_1)\wedge U_1^E(x_2)\wedge U_2^E(y_2)\,\vee\,
S_1(y_1,x_1)\wedge U_1^E(y_2)\wedge U_2^E(x_2)
\end{align*}
The graph $I(\mathbf A_n)$  is the half graph on $2n$ vertices (see Fig.~\ref{fig:half}).
\begin{figure}[ht]
\centering
	\includegraphics[width=.5\textwidth]{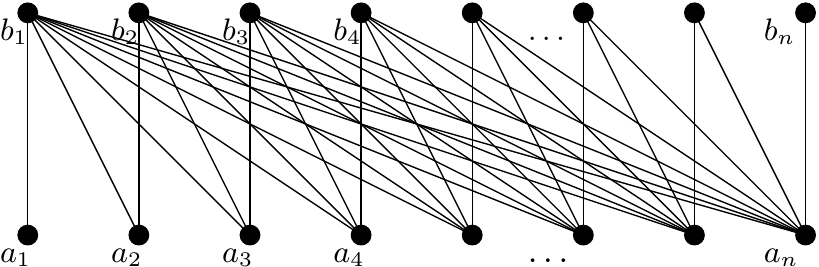}
\caption{Half graphs form a strongly polynomial sequence}\label{fig:half}
\end{figure}


\subsubsection{Intersection graphs of chords}

\begin{figure}[ht]
\centering
\begin{subfigure}{.5\textwidth}
  \centering
  \includegraphics[width=.65\linewidth]{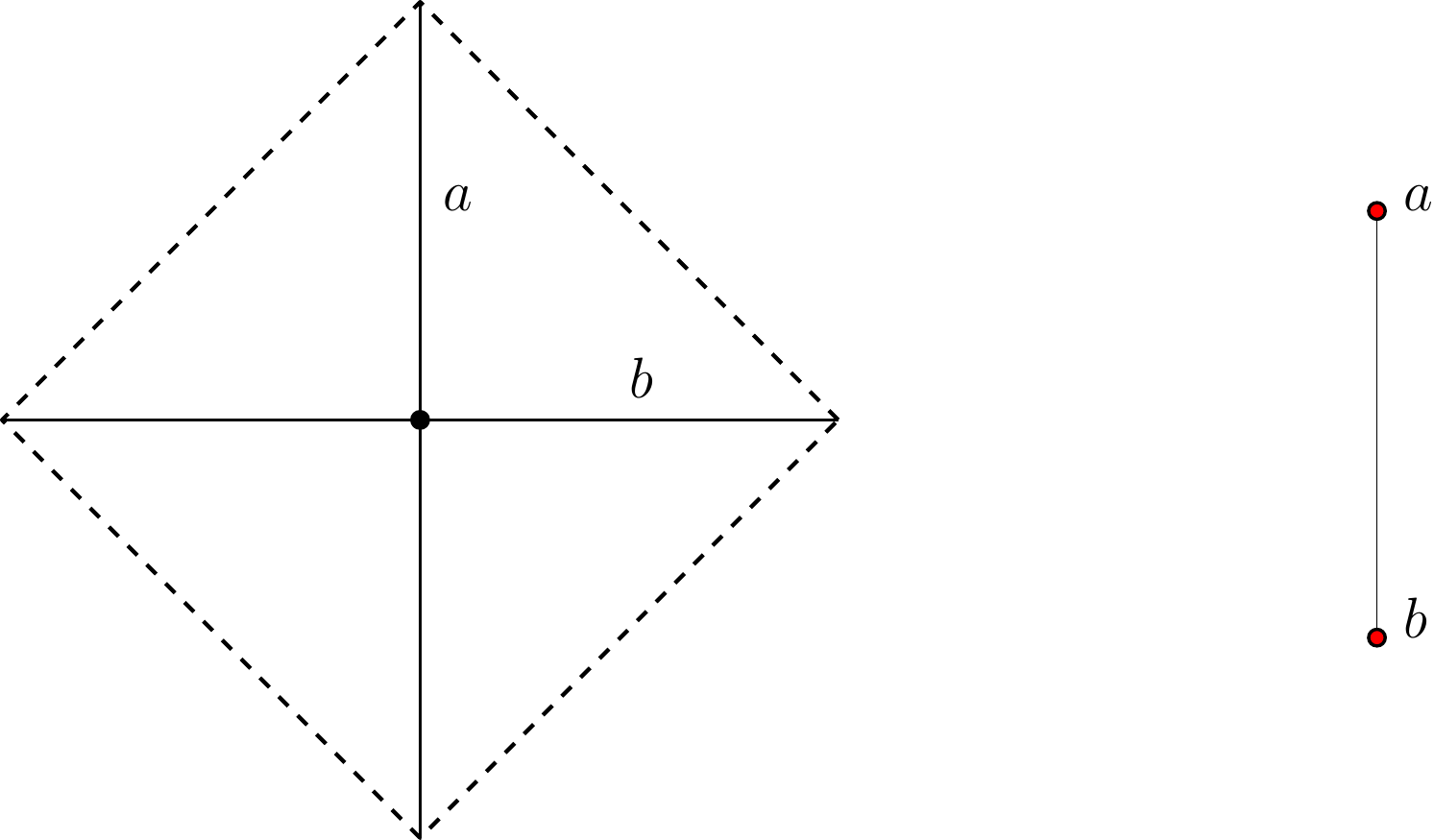}
  \caption{Square}
  \label{fig:D4}
\end{subfigure}%
\begin{subfigure}{.5\textwidth}
  \centering
  \includegraphics[width=.65\linewidth]{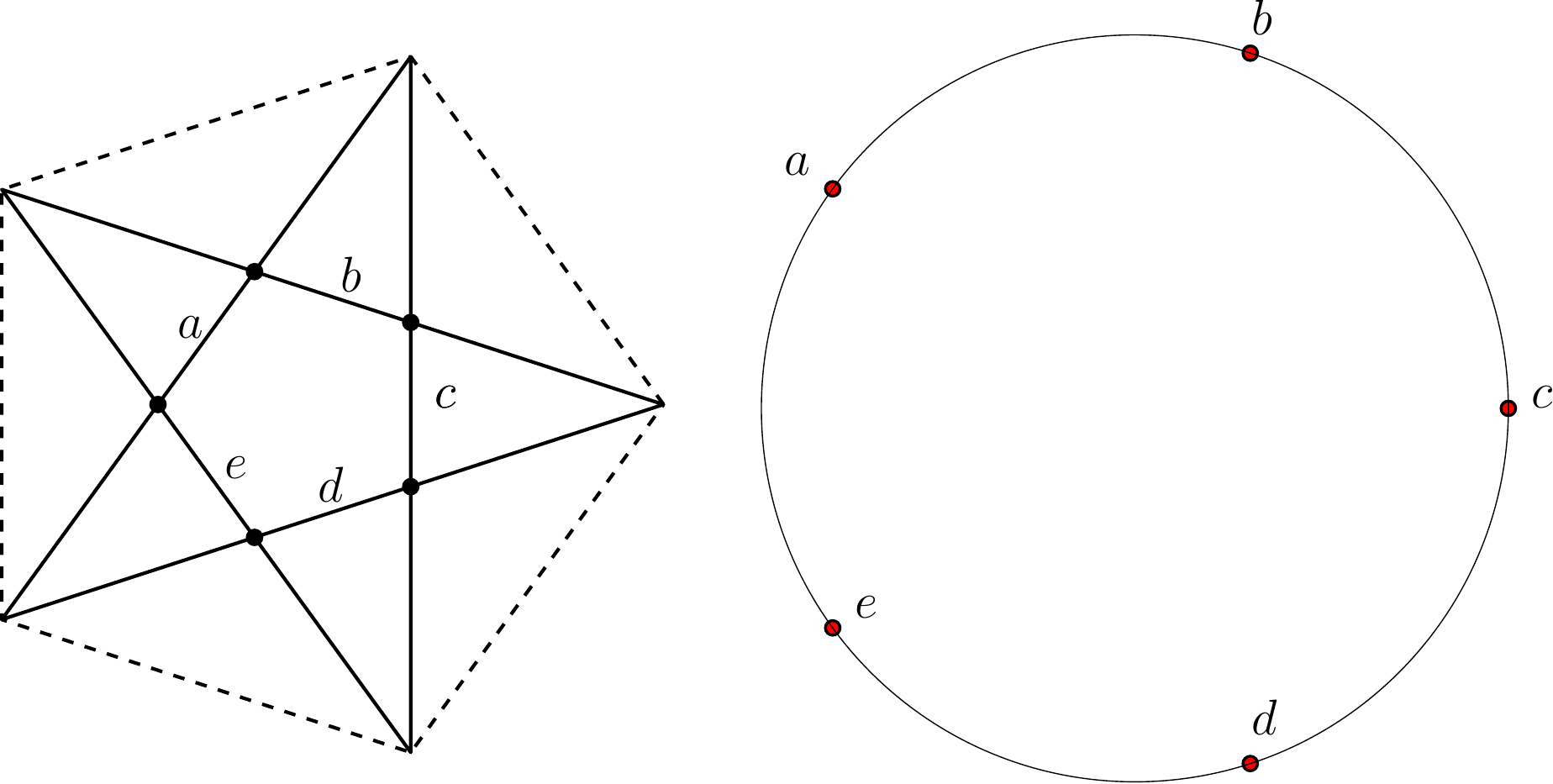}
  \caption{Pentagon}
  \label{fig:D5}
\end{subfigure}%
\vspace{3mm}

\centering\begin{subfigure}{.5\textwidth}
  \centering
  \includegraphics[width=.8\linewidth]{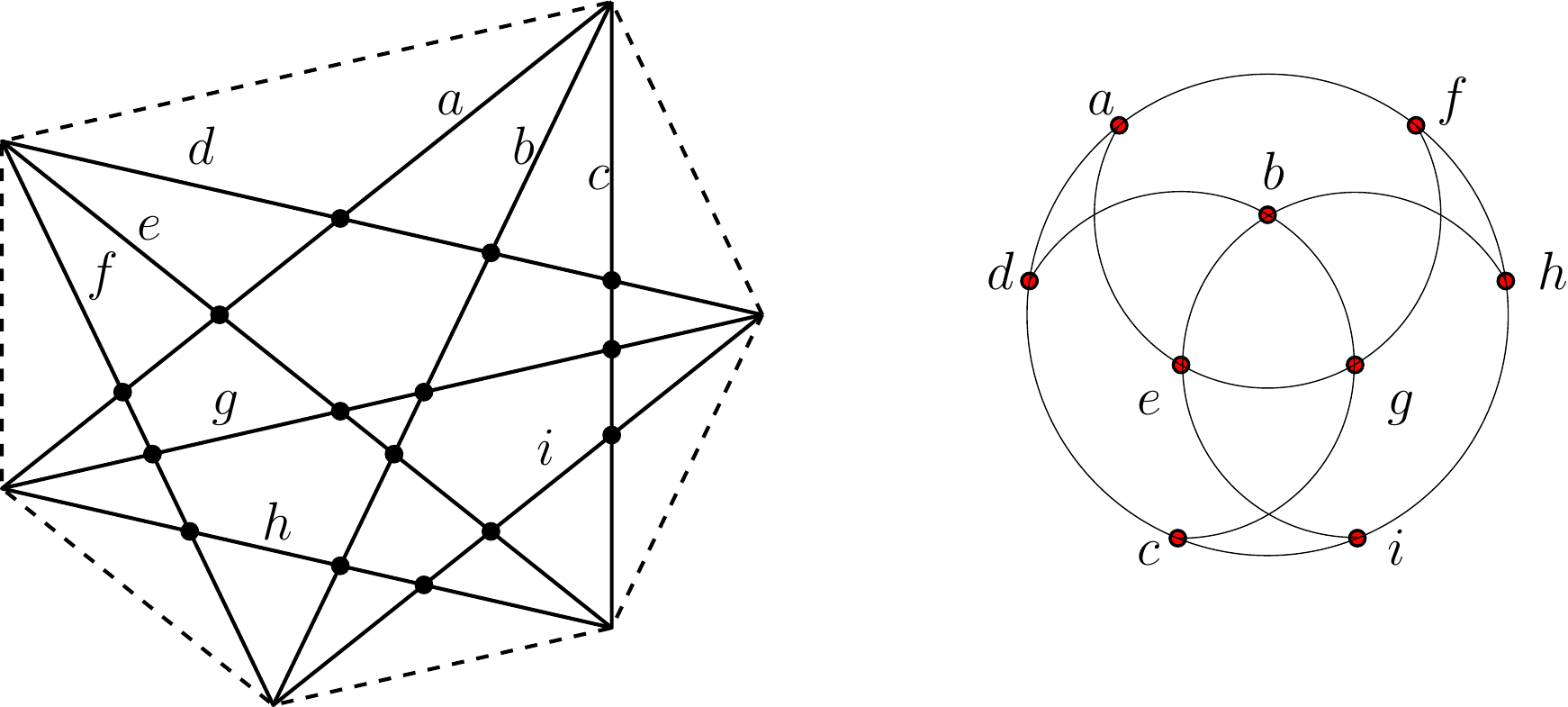}
  \caption{Hexagon}
  \label{fig:D6}
\end{subfigure}%
\begin{subfigure}{.5\textwidth}
  \centering
  \includegraphics[width=.8\linewidth]{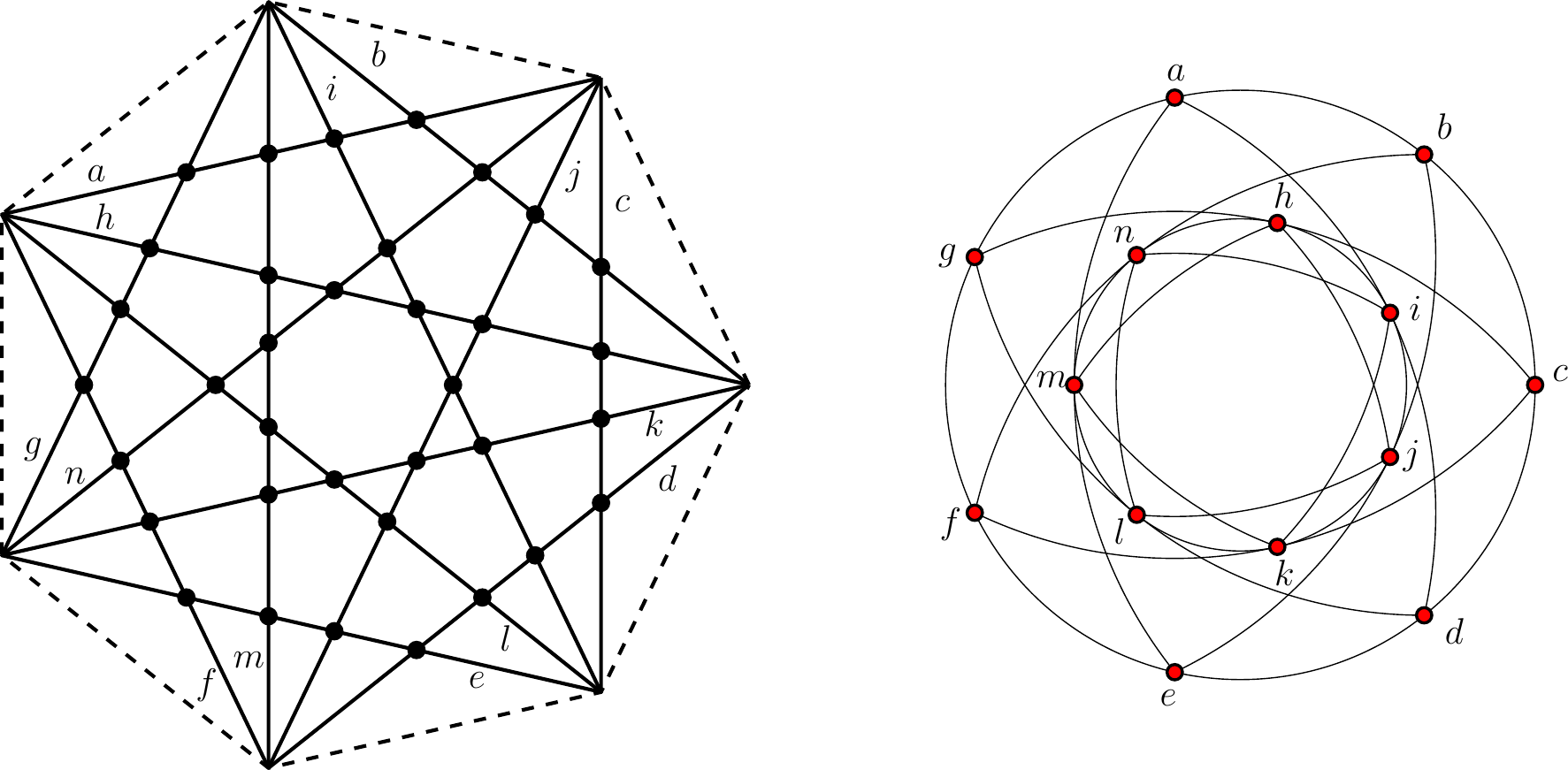}
  \caption{Heptagon}
  \label{fig:D7}
\end{subfigure}
\caption{Intersection graphs of chords of a convex $n$-gon form a strongly polynomial sequence}
\label{fig:chords}
\end{figure}

Let $\mathbf A_n=\mathbf T_n$.
Consider the graphical interpretation scheme $I=(2,\iota,\rho)$ where:
\begin{align*}
\iota(x_1,x_2):\quad&S_1(x_1,x_2)\\
\rho(x_1,x_2,y_1,y_2): 
&S_1(x_1,y_1)\wedge S_1(y_1,x_2)\wedge S_1(x_2,y_2)
\end{align*}
The graph $I(\mathbf A_n)$  is the intersection graph of chords of a convex $n$-gon (see Fig.~\ref{fig:chords}).

\subsection{Sequences of bounded degree graphs}
In this section we completely characterize strongly polynomial sequences of graphs of uniformly bounded degree.

\begin{theorem}
\label{thm:bounded}
Let $(\mathbf A_n)_{n\in\bbbn}$ be a sequence of graphs of uniformly bounded degree. Then the following conditions are equivalent:
\begin{enumerate}
\item  the sequence $(\mathbf A_n)_{n\in\bbbn}$ is strongly polynomial;
\item there is
a finite set $\{\mathbf F_1,\dots,\mathbf F_k\}$ of graphs and polynomials $P_1,\dots,P_k$ such that
$$\mathbf A_n=\sum_{i=1}^k P_i(n)\,\mathbf{F}_i;$$
\item the sequence $(\mathbf A_n)_{n\in\bbbn}$ is a
QF-interpretation of a basic sequence.
\end{enumerate}
\end{theorem}
\begin{proof}
Assume that the sequence $(\mathbf A_n)_{n\in\bbbn}$ is strongly polynomial. Let $P(n)=|A_n|$, $d=\deg P$ and, for a graph $\mathbf{F}$, let $P_{\mathbf{F}}(n)$ be the number of copies of $\mathbf F$ in $\mathbf A_n$. As $\Delta(\mathbf A_n)\leq D$ for some fixed bound $D$, we have $P_{\mathbf{F}}(n)\leq D^{|F|-1}\,P(n)$ and so all the polynomials
$P_{\mathbf{F}}$ have degree at most $d$. It follows
that $P_{\mathbf{F}}\neq 0$ if and only if there exists $i\leq d+1$ such that $P_{\mathbf{F}}(i)\neq 0$, that is, if and only if
$\mathbf{F}$ is an induced subgraph of $\bigcup_{i=1}^{d+1}\mathbf A_i$. Let $\mathbf F_1,\dots,\mathbf F_k$ be the connected induced subgraphs of $\bigcup_{i=1}^{d+1}\mathbf A_i$. As every connected component of $\mathbf A_n$ belongs to $\{\mathbf F_1,\dots,\mathbf F_k\}$, we infer that there exist polynomials $P_1,\dots,P_k$ such that
$\mathbf A_n=\sum_{i=1}^k P_i(n)\,\mathbf{F}_i$, as can be proved by induction on $k$ as follows. Let $\mathbf F$ be a connected maximal induced subgraph of $\bigcup_{i=1}^{d+1}\mathbf A_i$. Without loss of generality, we can assume $\mathbf F=\mathbf F_k$.
Let $P_k=P_{\mathbf{F}_k}$.
Then $\mathbf A_n$ contains
$P_{k}(n)$ copies of $\mathbf F_k$, each of them being a connected component of $\mathbf A_n$ (by maximality of $\mathbf{F}_k$). Hence we can define the sequence $(\mathbf{B}_n)_{n\in\bbbn}$ by requiring that 
$\mathbf{A}_n=\mathbf{B}_n+P_{k}(n)\,\mathbf F_k$.
The sequence $(\mathbf{B}_n)_{n\in\bbbn}$ is obviously strongly polynomial. Moreover, the connected induced subgraphs  of $\bigcup_{i=1}^{d+1}\mathbf B_i$ form a proper subset of the set of the connected induced subgraphs of $\bigcup_{i=1}^{d+1}\mathbf A_i$.
Without loss of generality, these induced subgraphs are $\mathbf F_1,\dots,\mathbf F_\ell$ (for some $\ell<k$). Hence, by induction hypothesis, there are polynomials $P_1,\dots,P_\ell$ such that
$\mathbf B_n=\sum_{i=1}^\ell P_i(n)\,\mathbf{F}_{i}$. Thus
$\mathbf A_n=\sum_{i=1}^\ell P_i(n)\,\mathbf{F}_{i}\,+P_{k}(n)\,\mathbf F_k$.

Assume that there is
a finite set $\{\mathbf F_1,\dots,\mathbf F_k\}$ of graphs and polynomials $P_1,\dots,P_k$ such that
$\mathbf A_n=\sum_{i=1}^k P_i(n)\,\mathbf{F}_i$.
Then the sequence $(\mathbf A_n)_{n\in\bbbn}$ is 
obviously a
QF-interpretation of a basic sequence.

Assume that the sequence $(\mathbf A_n)_{n\in\bbbn}$ is  a
QF-interpretation of a basic sequence. Then by Corollary~\ref{cor:intBasic} it is strongly polynomial, .
\end{proof}

\section{Left limits of strongly polynomial sequences}

Lov\'asz and Szegedy \cite{LS10} 
define a graph property (or equivalently a class of graphs) $\mathcal C$ 
to be {\em random-free} if every left limit of graphs in $\mathcal C$ is random-free. They prove the following:
\begin{theorem}[Lov\'asz and Szegedy \cite{LS10}]
\label{thm:RFLS}
A hereditary class $\mathcal C$ is random-free if and only if there exists a bipartite graph $F$ with bipartition
$(V_1,V_2)$ such that no graph obtained from $F$ by adding edges within $V_1$ or within $V_2$ is in $\mathcal C$.
\end{theorem}

This theorem has, in our setting, the following corollary, which gives a necessary condition for a sequence of graphs to be strongly polynomial.
\begin{theorem}\label{thm:random_free}
\label{thm:rf}
Every strongly polynomial sequence of graphs $(G_n)_{n\in\bbbn}$ converges to a random-free graphon.
\end{theorem}

\begin{proof}
Consider 
a strongly polynomial sequence of graphs $(G_n)_{n\in\bbbn}$. Let $P(n)=|G_n|$ and $d=\deg P$.
For every graph $F$, the probability that a random map
from $F$ to $G_n$ is a homomorphism is a fixed rational function of $n$, hence converges as $n\rightarrow\infty$.
It follows that the sequence $(G_n)_{n\in\bbbn}$ converges to some graphon $W$.

For $k\in\mathbb{N}$, consider the bipartite graph $F_k=(V_1,V_2,E)$, where
$|V_1|=k$, $|V_2|=2^{k}$, and the neighbourhoods of vertices in $V_2$ are pairwise distinct. 
 Let $F'$ be any graph obtained from $F_k$ by adding some edges whose endpoints both belong to $V_1$ or both to $V_2$.  
There are $\binom{2^k}{k}$ ways to choose $k$ vertices from $V_2$, which together with the $k$ vertices of $V_1$ induce a subgraph of $F'$ of order $2k$, which is unique up to the choice of the ordered part of $k$ vertices corresponding to $V_1$. 
Hence 
there are at least
$\binom{2^k}{k}/(k!\binom{2k}{k})=2^{k^2(1-o(1))}$ distinct induced subgraphs of $F'$ of order $2k$.
Thus, if a hereditary class $\mathcal C$ of graphs is not random-free, there exists for every integer $k$, according to Theorem~\ref{thm:RFLS}, a graph $F'$ derived from $F_k$ that belongs to $\mathcal C$. Hence the number of graphs of order $2k$ in $\mathcal C$ is at least $2^{k^2(1-o(1))}$.

To the sequence $(G_n)$ corresponds a hereditary class $\mathcal{F}=\{F:\exists n\; F\subseteq_i G_n\}$, consisting of graphs $F$ that occur as an induced subgraph of some $G_n$. 
If a graph $F$ of order $k$ belongs to $\mathcal F$, then it is an induced subgraph of a graph $G_n$ with $n\leq kd+1$. Indeed, the degree of the polynomial $P_F$ counting $F$ is at most $kd$, hence if
$P_F(n)=0$ for every $n\leq kd+1$, then $P_F=0$. 
It follows that the number of induced subgraphs of order $k$ is bounded by 
$\sum_{i=1}^{kd+1}\binom{P(i)}{k}=2^{o(k^2)}$.
It follows that every strongly polynomial sequence 
converges to a random-free graphon.
\end{proof}

The converse of the implication stated by Theorem~\ref{thm:rf} does not hold in general. Indeed, the set of strongly polynomial sequences is not closed under the operation of subsequence extraction, while the set of sequences converging to a random-free graphon does have this property.

\section{Going further}
\label{sec:Further}
We have seen that QF-interpretations of basic sequences form strongly polynomial sequences. We now extend this construction to generalized basic sequences, as a way to generate new strongly polynomial sequences from old.

\begin{definition}
Let $(\mathbf A_n)_{n\in\bbbn}$ be a sequence of $\lambda$-structures. Let $\lambda^+$ be the signature obtained from $\lambda$ by adding a new binary relation $S$ and a unary relation symbol $U$.

For $n\in\bbbn$, the  $\lambda^+$-structure
$\mathbf T \langle\mathbf A\rangle_n$ is obtained from the disjoint union
$\sum_{i=1}^n\mathbf A_n$ by adding relations $S$ and $U$ as follows:
\begin{itemize}
\item for every vertex $x$, $\mathbf T \langle\mathbf A\rangle_n\models U(x)$;
\item for all vertices $x\in A_i$ and $y\in A_j$, 
 $\mathbf T \langle\mathbf A\rangle_n\models S(x,y)$ if $i<j$.
\end{itemize}
\end{definition}

\begin{lemma}
Let $(\mathbf A_n)_{n\in\bbbn}$ be a strongly polynomial sequence of $\lambda$-structures. Let $\lambda^+$ be the signature obtained from $\lambda$ by adding a new binary relation $S$ and a unary relation symbol $U$.

Then $(\mathbf T \langle\mathbf A\rangle_n)_{n\in\bbbn}$ is a strongly polynomial sequence of $\lambda^+$-structures.
\end{lemma}
\begin{proof}
Let $\mathbf F$ be a connected $\lambda^+$-structure.
Say that 
  $\mathbf F$ is nice if $F$ can be partitioned as $F=\bigsqcup_{i=1}^k F_i$, with the property that for every $(x,y)\in F_i\times F_j$ it holds that $\mathbf F\models S(x,y)$ if and only if $i<j$ and  $\mathbf F\models (\forall x) U(x)$. Then the number ${\rm inj}(\mathbf F,\mathbf T \langle\mathbf A\rangle_n)$ of injective homomorphisms $f:\mathbf F\rightarrow \mathbf T \langle\mathbf A\rangle_n$ is given by
  $$
  {\rm inj}(\mathbf F,\mathbf T \langle\mathbf A\rangle_n)=\begin{cases}
  0,&\text{if }\mathbf F\text{ is not nice,}\\
  \displaystyle\sum_{1\leq i_1<\dots<i_k\leq n}\ \prod_{j=1}^k {\rm inj}(\mathbf F_j,\mathbf A_{i_j}),&\text{otherwise,}
  \end{cases}
  $$
  where $\mathbf F_j$ is the substructure of $\mathbf F$ induced on $F_j$.
As $(\mathbf A_n)_{n\in\bbbn}$ is strongly polynomial, there are polynomials $P_1,\dots,P_k$ such that 
$${\rm inj}(\mathbf F_j,\mathbf A_{n})=P_j(n).$$
For every $n\in\bbbn$, we have
$$\sum_{1\leq i_1<\dots<i_k\leq n}\ \prod_{j=1}^k P_j(i_j)=
\sum_{i_1=1}^n P_1(i_1) \left(\sum_{i_2=i_1+1}^n P_2(i_2) \left(\dots \sum_{i_k=i_{k-1}+1}^n P_k(i_k)\dots\right)\right).$$
But for each $k$ there exists a polynomial $Q_k$ such that 
 $\sum_{i_k=i_{k-1}+1}^n P_k(i_k)=Q_k(n)-Q_k(i_{k-1})$, in which $i_0=0$. 
 By induction on $k$, it follows that there exists a polynomial $Q_{\mathbf F}$ such 
 that for every $n\in\bbbn$ we have
  $$
 Q_{\mathbf F}(n)=\begin{cases}
  0,&\text{if }\mathbf F\text{ is not nice,}\\
  \displaystyle\sum_{1\leq i_1<\dots<i_k\leq n}\ \prod_{j=1}^k P_j(i_j),&\text{otherwise.}
  \end{cases}
  $$
It follows that the sequence $(\mathbf T \langle\mathbf A\rangle_n)_{n\in\bbbn}$ is strongly polynomial.
\end{proof}
\begin{definition}\label{def:gen_basic}
A {\em generalized basic structure with parameter }
$$(((\mathbf A^1_n)_{n\in\bbbn},\dots,(\mathbf A^k_n)_{n\in\bbbn}),(\mathbf B^1,\dots,\mathbf B^\ell))$$
 is any structure of the form
$$
\mathbf C=\bigoplus_{i=1}^\ell \mathbf B^i\oplus\bigoplus_{j=1}^k \mathbf{T}\langle\mathbf{A}^j\rangle_{N_j},$$
with $N_1,\dots,N_k\in\bbbn$. These integers will be also be denoted by 
$N_1(\mathbf C),\dots,N_k(\mathbf C)$.

A {\em generalized basic sequence} 
 is a sequence $(\mathbf C_n)_{n\in\bbbn}$ of generalized basic structures 
$\mathbf C_n$ with the same parameter $(((\mathbf A^1_n)_{n\in\bbbn},\dots,(\mathbf A^k_n)_{n\in\bbbn}),(\mathbf B^1,\dots,\mathbf B^\ell))$,
such that 
there are non-constant polynomials $Q_i$, $1\leq i\leq k$ with $Q_i(n)=N_i(\mathbf{C}_n)$
(for every $1\leq i\leq k$ and $n\in\bbbn$).
\end{definition}

\begin{theorem}
For every generalized basic sequence $(\mathbf C_n)_{n\in\bbbn}$ and every
QF-interpretation scheme $I$, the sequence $(I(\mathbf C_n))_{n\in\bbbn}$
is strongly polynomial.
\end{theorem}
\begin{proof}
As the sequence $(\mathbf T \langle\mathbf A\rangle_n)_{n\in\bbbn}$ is strongly polynomial, 
this theorem is a direct consequence of Lemmas~\ref{lem:extractP} and~\ref{lem:union}.
\end{proof}

\section{Discussion}\label{sec:Discussion}
The notion of interpretation scheme introduced in this paper could be made even more general by introducing a formula $\varpi$ (with $2p$ free variables) defining an equivalence relation on the set of $p$-tuples compatible with the formulas $\rho_i$, in the sense that
$$\bigwedge_{j=1}^{r_i}\varpi(\mathbf x_j,\mathbf y_j)\quad\vdash\quad
\rho_i(\mathbf x_1,\dots,\mathbf x_{r_i})\leftrightarrow
\rho_i(\mathbf y_1,\dots,\mathbf y_{r_i})$$
 (see for instance~\cite{L09}).
For a general interpretation scheme $I$ of $\lambda$-structures in $\kappa$-structures and a $\kappa$-structure $\mathbf A$, the vertex set of $I(\mathbf A)$ is then the set of $\varpi$-equivalence classes $[\mathbf x]$ of $p$-tuples $\mathbf x$ such that $\mathbf A$ satisfies $\rho_0(\mathbf x)$.

In such a context, we can prove the following generalization of Corollary~\ref{cor:IntPreserveSP}.

\begin{theorem}
\label{thm:quot}
Let $(\mathbf A_n)_{n\in\bbbn}$ be a strongly polynomial sequence of $\kappa$-structures and let
$I=(p,\varpi,\rho_0,\dots,\rho_k)$ 
 be a general interpretation scheme of $\lambda$-structures in $\kappa$-structures such that
 all the formulas $\varpi,\rho_0,\dots,\rho_k$ are quantifier-free.
 
 Assume that all the $\varpi$-equivalence classes have polynomial quantifier-free definable cardinalities. More precisely, we assume that there exists an integer $N$,
 polynomials $Q_1,\dots,Q_N$ (such that $Q_i(n)\in\mathbb{N}$ for every $1\leq i\leq N$ and every $n\in\bbbn$), 
and quantifier-free formulas $\eta_1,\dots,\eta_N$ (with $p$ free variables) such that $\bigvee_{i=1}^N \eta_i=1$ and for every $n\in\bbbn$ and every $(v_1,\dots,v_p)\in A_n^p$ the $\varpi$-equivalence class $[(v_1,\dots,v_p)]$
of $(v_1,\dots,v_p)$ in $A_n^p$ has cardinality exactly $Q_i(n)$ if $\mathbf A_n\models\eta_i(v_1,\dots,v_p)$.

Then 
$(I(\mathbf A_n))_{n\in\bbbn}$ is a strongly polynomial sequence of $\lambda$-structures.
\end{theorem}
\begin{proof}
Let $I'=(p,\rho_0,\dots,\rho_k)$ and let $\phi\in{\rm QF}_q(\lambda)$ be a quantifier-free formula.
According to Lemma~\ref{lem:dual}, for every $\kappa$-structure $\mathbf A$ we have $\phi(I'(\mathbf A))=\tilde{I'}(\phi)(\mathbf A)$.
For $f:[q]\rightarrow [N]$, define the quantifier free formula $\psi_f\in{\rm QF}_q(\lambda)$ as follows:
$$\psi_f(\mathbf x_1,\dots,\mathbf x_p):\quad
\tilde{I'}(\phi)(\mathbf x_1,\dots,\mathbf x_p)\wedge
\bigwedge_{i=1}^N\eta_{f(i)}(\mathbf x_i).$$
By hypothesis it is immediate that for every $n\in\bbbn$ we have
$$
|\phi(I(\mathbf A_n))|=\sum_{f:[q]\rightarrow [N]}
\frac{|\psi_f(\mathbf A_n)|}{\prod_{i=1}^q Q_i(n)}.
$$
Hence $|\phi(I(\mathbf A_n))|$ is a rational function of $n$, thus a polynomial function of $n$ (according to Remark~\ref{rem:rat}).
\end{proof}
General interpretation schemes allow the definition of new constructions preserving the property of a sequence being strongly polynomial. For instance:
\begin{itemize}
\item $G\mapsto \dot{G}$, which maps a graph to its $1$-subdivision (that is, the graph obtained from $G$ by replacing each edge by a path of length $2$);
\item $G\mapsto L(G)$, which maps a simple graph to its line graph.
\end{itemize}

However, if we restrict ourselves to the class of basic structures, introducing an equivalence relation with bounded class cardinalities does not in fact lead to the definition of any new transformations.
By a fine study of quantifier-free definable equivalence relations, one can prove the following result.



\begin{theorem}
\label{thm:simple}
Let $k,\ell$ be integers and let $I=(p,\varpi,\rho_0,\dots,\rho_k)$ be a general  interpretation scheme of $\lambda$-structures in $\beta_{k,\ell}$-structures defined by quantifier-free formulas such that there exists an integer $N$ with the property that, for every $\mathbf B\in\mathcal{B}_{k,\ell}$ and every $(v_1,\dots,v_p)\in B^p$, the $\varpi$-equivalence class of  
$(v_1,\dots,v_p)$ has cardinality at most $N$.

Then there exists a (restricted) QF-interpretation scheme
$\hat I=(\hat\rho_0,\rho_1,\dots,\rho_k)$ of $\lambda$-structures in $\beta_{k,\ell}$-structures 
such that for every
$B\in\mathcal B_{k,\ell}$, the $\lambda$-structures 
$I(\mathbf B)$ and $\hat{I}(\mathbf B)$ are isomorphic.
\end{theorem}
\begin{proof}[Proof sketch]
Consider the graphical interpretation scheme $I'=(p,\rho_0,\varpi)$. The condition that $\varpi$-equivalence classes have cardinality at most $N$ translates to the fact that  for every $\mathbf B\in\mathcal{B}_{k,\ell}$ the maximum degree of $I'(\mathbf B)$ is at most $N\!-\!1$. 

A {\em $(k,\ell)$-pattern of length $p$} is a surjective mapping 
$$F:[p]\rightarrow I\cup \bigcup_{j\in J} \{j\}\times [n_j],$$
where $I\subseteq [\ell]$, $[J]\subseteq [k]$, and $1\leq n_j\leq p$.
The {\em profile} of $F$ is the $k$-tuple $(n_1,\dots,n_k)$.
For $F(i)\notin I$, we denote by $F(i)_1$ and $F(i)_2$ the two 
coordinates of $F(i)$ (so that $F(i)=(F(i)_1,F(i)_2)$.

To a $(k,\ell)$-pattern $F$ of length $p$ we associate the quantifier-free formula

\begin{equation}
\begin{split}
\tau_F:\ &\bigwedge_{i: F(i)\in I}U_{F(i)}^E(x_i)\,\wedge\,
\bigwedge_{i: F(i)\notin I}U^T_{F(i)_1}(x_i)
\,\wedge\,\bigwedge_{i,j: F(i)=F(j)}(x_i=x_j)\\
&\,\wedge\,\bigwedge_{\substack{i,j: F(i),F(j)\notin I\\\text{and }F(i)_1=F(j)_i\\\text{and } F(i)_2<F(j)_2}}S_{F(i)_1}(x_i,x_j).
\end{split}
\end{equation}

Then it is easily checked that for  $k,\ell\in\bbbn$ and every $\mathbf A\in\mathcal{B}_{k,\ell}$ there exists a unique $(k,\ell)$-pattern $F_{\mathbf{A}}$ such that $\mathbf A\models \tau_{F_{\mathbf{A}}}(\mathbf{A})$.
Also, for every quantifier-free formula $\phi\in{\rm QF}_p$ there exists a finite family $\mathcal F$ of $(k,\ell)$-patterns such that for every $\mathbf A\in\mathcal{B}_{k,\ell}$ and every $v_1,\dots,v_p\in A$ it holds that
$$\mathbf A\models\phi(v_1,\dots,v_p)\quad\iff\quad
\mathbf A\models\bigvee_{F\in\mathcal{F}}\tau_F(v_1,\dots,v_p).$$

To every $p$-tuple $\mathbf v=(v_1,\dots,v_p)\in A^p$, with $\mathbf{A}\in\mathcal{B}_{k,\ell}$, we associate
the vector $\overline{\mathbf v}$ of distinct coordinates of $\mathbf{v}$ belonging to tournaments, taken in order. A vector $\mathbf{u}$ is {\em packed} if $\mathbf{u}=\overline{\mathbf{u}}$. The profile of a vector $\mathbf{u}$ is the profile of $F_{\mathbf{u}}$. Note that $\mathbf u$ and $\overline{\mathbf u}$ have the same profile.
 Note also that 
$\zeta:{\mathbf v}\mapsto (\overline{\mathbf v}, F_{\mathbf v})$ is a bijection between $A^p$ and pairs $(\mathbf u,F)$ such that $\mathbf u$ is packed and has the same profile as $F$. For given $F$, we denote by $\Omega(\mathbf A, F)$ the packed vectors with coordinates in $A$ having the same profile as $F$.
For $k,\ell\in\bbbn$ and  $\phi\in{\rm QF}_p(\beta_{k,\ell})$, there exists a set $\mathcal{F}$ of $(k,\ell)$-patterns, such that 
$$\phi(\mathbf A)=\bigcup_{F\in\mathcal{F}} 
\{\zeta^{-1}(\mathbf u, F):\ \mathbf u\in\Omega(\mathbf A, F)\},$$
where the union is a disjoint union.

From the hypothesis that $\varpi$ is an equivalence relation with bounded classes, we deduce that there exists a collection $\{\mathcal F_1,\dots,\mathcal F_t\}$ of disjoint sets of $(k,\ell)$-patterns with the same profile, such that each connected component of $I'(\mathbf B)$ is a clique with vertex set 
$$K_{t,\mathbf x}=\{\zeta^{-1}(\mathbf x, F): F\in\mathcal F_t\},$$
where $\mathbf x$ is any compressed vector with the same profile as $F\in\mathcal F_t$. (This is shown by proving that for an equivalence relation $\varpi$ of a different form we can construct $\mathbf A$ with at least one arbitrarily large equivalence class).

Consider an arbitrary linear order on $(k,\ell)$-patterns.
Let $\hat\rho_0\in{\rm QF}_p$ be the quantifier-free formula
$$\hat\rho_0(x_1,\dots,x_p):\quad\rho_0(x_1,\dots,x_p)\wedge\bigvee_{i=1}^t\tau_{\min \mathcal F_i}(x_1,\dots,x_p).$$
Then $\hat\rho_0$ selects exactly one vertex in each $\varpi$-equivalence class, and the statement of the lemma now follows.
\end{proof}
We do not include the full lengthy technical proof of Theorem~\ref{thm:simple}.
\section{Conclusion and open problems}
A natural problem arising from this paper is to
figure out whether the strongly polynomial sequences we have constructed from basic sequences constitute the general case, as suggested by Theorem~\ref{thm:bounded}.

A strongly polynomial sequence of graphs $(G_n)_{n\in\bbbn}$ is {\em induced monotone} if $G_n$ is an induced subgraph of $G_{n+1}$ for each $n\in\bbbn$.
Equivalently, $(G_n)_{n\in\bbbn}$ is induced monotone if there exists
a countable graph $G$ such that $G_n$ is the subgraph induced by vertices
$1,\dots,f(n)$, where $f$ is a monotone (non-decreasing) function.

\begin{problem} Can all strongly polynomial induced monotone sequences of graphs be obtained by QF-interpretation schemes from generalized basic sequences (as defined in Definition~\ref{def:gen_basic})?
\end{problem} 

There is more to this than meets the eye. By Theorem~\ref{thm:rf}, we know that every strongly polynomial sequence converges to a random-free graphon. In some sense, we are asking here whether, under the stronger assumption that the sequence of graphs has an inductive countable limit, the countable limit itself may be an interpretation of a countable ``basic" structure.

%

%

The Paley graph indexed by $q$ a prime power congruent to $1$ modulo $4$ is defined as the Cayley graph on $\mathbb{F}_q$ with edges joining vertices whose difference is a non-zero square in $\mathbb{F}_q$. Paley graphs form a quasi-random sequence and so as a consequence of Theorem~\ref{thm:random_free} any strongly polynomial sequence of graphs can contain only finitely many distinct Paley graphs.
Nonetheless, a corollary of~\cite[Prop. 8]{dlHJ95} is that when $F$ is a 2-degenerate graph (for example, a series-parallel graph) the number of homomorphic images of $F$ in the Paley graph on $q$ vertices is a polynomial in $q$ (dependent only on $F$). 

In~\cite{GGN} it is shown that for the hypercube $Q_n\!=\!{\rm Cayley}(\mathbb{F}_2^n, S_1)$, where $S_1$ is the set of $n$ vectors of Hamming weight $1$, for each $F$ the quantity ${\rm hom}(F,Q_n)$ is a polynomial in $2^n$ and~$n$ (dependent only on $F$).

These examples prompt the following question concerning sequences of Cayley graphs:
\begin{problem}\label{prob:Cayley}
Let $(A_n)$ be a sequence of groups and $(B_n)$ a sequence of subsets, $B_n\subseteq A_n$, that are closed under inverses.
When is it the case that 
${\rm hom}(F,{\rm Cayley}(A_n,B_n))$ is a fixed bivariate polynomial (dependent on $F$) in $|A_n|$ and $|B_n|$ for sufficiently large $n\geq n_0(F)$?
\end{problem}
 If here $|A_n|$ and $|B_n|$ are both polynomial in $n$ then this is to ask whether the sequence $({\rm Cayley}(A_n,B_n)\,)$ is a polynomial sequence of graphs (see Section~\ref{sec:intro}).
\begin{remark} By a variation of the argument given in{\rm ~\cite[Prop.2]{dlHJ95}}, when $A_n=\mathbb{Z}_n$ and $B_n=B\cap [-(n-1), n-1]$ for $B\subseteq\mathbb{Z}$, $B=-B$, the sequence $({\rm Cayley}(A_n,B_n))$ is polynomial  if and only if $B$ is finite or cofinite.
\end{remark}

Is it true for instance that the sequence $({\rm Cayley}(\mathbb{Z}_n,D_n))$ where $D_n=[a_n,b_n]$ is an interval of length $o(n)$ is a polynomial sequence?
 
 In another direction, perhaps one should try to approach by means of interpretations other (multivariate) polynomials, such as the Bollob\'as-Riordan coloured Tutte polynomial \cite{BR99}.
\bibliographystyle{amsplain}

\end{document}